\documentclass[11pt]{amsart}

\usepackage[T1]{fontenc}
\usepackage{lmodern}
\usepackage{microtype}
\usepackage{amsmath,amssymb,mathtools}
\usepackage{enumitem}
\usepackage[colorlinks=true,linkcolor=blue,citecolor=blue,urlcolor=blue]{hyperref}

\setlist[enumerate]{label=(\arabic*),leftmargin=2.2em}

\newtheorem{theorem}{Theorem}[section]
\newtheorem{lemma}[theorem]{Lemma}
\newtheorem{proposition}[theorem]{Proposition}
\newtheorem{corollary}[theorem]{Corollary}
\theoremstyle{definition}

\theoremstyle{remark}
\newtheorem{remark}[theorem]{Remark}

\newcommand{\B}{\mathbb B}
\newcommand{\Z}{\mathbb Z}
\newcommand{\SL}{\operatorname{SL}}
\newcommand{\Gen}{\operatorname{Gen}}
\newcommand{\Rand}{\operatorname{Rand}}
\newcommand{\restr}{\mathbin{\upharpoonright}}
\newcommand{\EBM}{E^{\B}_{M}}
\newcommand{\eqgen}[2]{\mathrel{\equiv^{#2}_{#1}}}
\newcommand{\leBw}{\mathrel{\leq^{w}_{B}}}

\title[Equivalence of random generics]
{Equivalence of random generics is not essentially free}
\author{Jason Zesheng Chen$^\ast$}
\thanks{$^\ast$Independent Researcher, San Jose, CA, USA}
\date{August 21, 2026}

\hypersetup{%
  pdftitle={Equivalence of random generics is not essentially free},
  pdfauthor={Jason Zesheng Chen},
  pdfsubject={Random forcing and countable Borel equivalence relations},
  pdfkeywords={random forcing, countable Borel equivalence relations, essential freeness, Borel reducibility, Bernoulli shifts, cocycle superrigidity}
}

\begin{document}

\begin{abstract}
Let $M$ be a countable transitive model of a sufficiently large finite fragment of ZFC, and let $E^{\B}_{M}$ be equivalence of random generics over $M$.  Smythe proved that $E^{\B}_{M}$ is not Borel-reducible to the orbit relation of a free action of any countable group belonging to $M$, and asked whether the restriction on the target group can be removed.

We prove that if $A$ is a positive-measure Borel set of random reals over $M$ and $F$ is an essentially free countable Borel equivalence relation, then every Borel homomorphism from $E^{\B}_{M}\restr A$ to $F$ maps a conull subset of $A$ into a single $F$-class.  Hence no positive-measure restriction of $E^{\B}_{M}$ is essentially free, or even weakly Borel-reducible to an essentially free relation.  The proof combines Thomas's consequence of Popa cocycle superrigidity with a perfect family of $2$-marked groups and a Fubini argument using two successive random reals.
\end{abstract}

\maketitle
\makeatletter
\global\topskip\normaltopskip
\makeatother

\section{Introduction}

Throughout, $M$ is a countable transitive model of a fixed finite fragment of ZFC, taken sufficiently large for the forcing constructions below.  For a forcing notion $\mathbb P\in M$, equivalence of $M$-generic filters is the relation
\[
 G\eqgen{M}{\mathbb P}H
 \quad\Longleftrightarrow\quad
 M[G]=M[H].
\]
For $\mathbb P\in M$, let $\Gen_M^{\mathbb P}$ denote the space of $M$-generic filters, viewed as a subspace of $2^{\mathbb P}$.  Smythe showed that $\Gen_M^{\mathbb P}$ is a $G_\delta$ subspace of $2^{\mathbb P}$, and hence Polish, and that $\eqgen{M}{\mathbb P}$ is a countable Borel equivalence relation on this space~\cite[Lemmas~2.2 and~2.6]{Smythe}.  Indeed, the equivalence class of $G$ is contained in the countable model $M[G]$, while Borelness follows from the definability of the forcing relation.  Smythe then studied its complexity for several classical forcing notions.  This line of investigation has recently been continued by Calderoni and Sinapova~\cite{CalderoniSinapova}.

\newpage
\mbox{}\par
\vspace{\baselineskip}
For random forcing $\B$, it is convenient to work with random reals.  Write
\[
 \Rand(M)=\{x\in 2^\omega:x\text{ is random over }M\}
\]
and, for $x,y\in\Rand(M)$, put
\[
 x\,\EBM\,y
 \quad\Longleftrightarrow\quad
 M[x]=M[y].
\]
The set $\Rand(M)$ is a conull Borel set, and $\EBM$ is a countable Borel equivalence relation.

Smythe proved that $\EBM$ is neither amenable nor treeable.  He also proved that it is not Borel-reducible to the orbit relation of a free Borel action of any countable group $\Gamma\in M$.  His proof uses a theorem of Thomas, derived from Popa's cocycle superrigidity: after the target group $\Gamma$ is fixed, one chooses in $M$ a finitely generated group that does not embed into $\Gamma$, and realizes an associated Bernoulli relation inside $\EBM$.  The hypothesis $\Gamma\in M$ is used precisely at this point.  Smythe asked whether $\EBM$ is essentially free without any restriction on the target group~\cite[Question~5.3]{Smythe}.

We answer this question negatively, in a stronger form.  Recall that a countable Borel equivalence relation $F$ is \emph{essentially free} if it is Borel reducible to the orbit relation of a free Borel action of a countable group.  A Borel map $f\colon X\to Y$ is a \emph{Borel homomorphism} from an equivalence relation $E$ on $X$ to an equivalence relation $F$ on $Y$ if
\[
 x\mathrel{E}x'\quad\Longrightarrow\quad f(x)\mathrel{F}f(x').
\]
Following Smythe's terminology, when $E$ is carried by a standard probability space $(X,\mu)$, such a homomorphism is \emph{$\mu$-trivial} if a conull Borel subset of $X$ is mapped into a single $F$-class.  We say that $(E,\mu)$ is \emph{$F$-ergodic} if every Borel homomorphism $E\to F$ is $\mu$-trivial.  We write $E\leBw F$ if there is a countable-to-one Borel homomorphism from $E$ to $F$; such a map is called a \emph{weak Borel reduction}.

For a Borel set $A\subseteq2^\omega$ of positive measure, let $\mu_A$ denote normalized measure on $A$.

\begin{theorem}\label{thm:main}
Let $A\subseteq\Rand(M)$ be Borel with $\mu(A)>0$.  If $F$ is an essentially free countable Borel equivalence relation, then $(\EBM\restr A,\mu_A)$ is $F$-ergodic.
\end{theorem}

Hjorth used the phrase \emph{highly unfree} for equivalence relations whose restriction to no conull set is essentially free~\cite{Hjorth}.  Theorem~\ref{thm:main} strengthens this in two directions: it applies on every positive-measure Borel set and to arbitrary Borel homomorphisms, not only reductions.

\begin{corollary}\label{cor:positive}
If $A\subseteq\Rand(M)$ is Borel and $\mu(A)>0$, then
\[
 \EBM\restr A\not\leBw F
\]
for every essentially free countable Borel equivalence relation $F$.  In particular, no positive-measure restriction of $\EBM$ is essentially free, and hence $\EBM$ is highly unfree.
\end{corollary}

Thus Corollary~\ref{cor:positive} strengthens the negative answer to Smythe's question.

The argument removes the hypothesis $\Gamma\in M$ as follows.  Fix an external countable target group $\Gamma$.  Inside $M$ we choose a perfect family $(L_x)_{x\in2^\omega}$ of $2$-marked groups such that only countably many $L_x$ embed into $\Gamma$.  Hence for almost every random real $x$ over $M$, one has $L_x\in M[x]$ but $L_x$ does not embed into $\Gamma$.  A second real, random over $M[x]$, realizes the relevant Bernoulli shift, and interleaving the two random reals produces a single random real over $M$.

Thomas's theorem then shows, for almost every first coordinate, that the homomorphism is essentially constant along the corresponding vertical section.  Writing $I(x,u)$ for the real obtained by interleaving $x$ and $u$, one has, for product-random pairs,
\[
 M[I(x,u)]=M[x,u]=M[I(u,x)],
\]
so the values at the two transposed coordinates lie in the same target orbit.  Fubini's theorem gives global measure-triviality.  Finally, the ergodicity of $\EBM$ passes from the conull case to every positive-measure restriction.

We also record an immediate forcing consequence.  If $\dot\B$ denotes the canonical name for random forcing in a forcing extension, then for every $\mathbb P\in M$ the relation $\eqgen{M}{\mathbb P*\dot\B}$ is not weakly Borel-reducible to any essentially free countable Borel equivalence relation; see Corollary~\ref{cor:random-tail}.

\section{Preliminaries}

Let $\mu$ be the usual product measure on $2^\omega$.  A real $x\in2^\omega$ is random over a countable transitive model $N$ if it belongs to every conull Borel set whose code belongs to $N$.  We write $\Rand(N)$ for the set of such reals, and let $E_N^{\B}$ be the equivalence relation on $\Rand(N)$ defined by
\[
 x\mathrel{E_N^{\B}}y
 \quad\Longleftrightarrow\quad
 N[x]=N[y].
\]
A pair $(x,u)\in(2^\omega)^2$ is \emph{product-random over $N$} if it is random over $N$ for the product measure.  For $A\subseteq2^\omega\times2^\omega$ and $v\in2^\omega$, write
\[
 A_v=\{u\in2^\omega:(v,u)\in A\}.
\]

Let
\[
 I\colon 2^\omega\times2^\omega\longrightarrow2^\omega
\]
be the interleaving homeomorphism
\[
 I(x,u)(2n)=x(n),\qquad I(x,u)(2n+1)=u(n).
\]
It is measure preserving, and its code belongs to every model under consideration.

\begin{lemma}\label{lem:two-step}
Suppose that $x$ is random over $N$ and $u$ is random over $N[x]$.  Then $(x,u)$ is product-random over $N$.  Consequently $I(x,u)$ is random over $N$ and
\[
 N[I(x,u)]=N[x,u].
\]
\end{lemma}

\begin{proof}
Let $A\subseteq2^\omega\times2^\omega$ be an $N$-coded Borel null set.  By Fubini's theorem,
\[
 B_A=\{v\in2^\omega:\mu(A_v)>0\}
\]
is an $N$-coded null Borel set.  Since $x$ is random over $N$, $x\notin B_A$, and hence $A_x$ is null.  A code for $A_x$ belongs to $N[x]$, so $u\notin A_x$.  Thus $(x,u)\notin A$.  Since $I$ is measure preserving and coded in $N$, the real $I(x,u)$ is random over $N$.  Since both $I$ and $I^{-1}$ are coded in $N$, one also has $N[I(x,u)]=N[x,u]$.
\end{proof}

Let $\Delta$ be a countably infinite group.  The Bernoulli shift $\Delta\curvearrowright2^\Delta$ is given by
\[
 (\delta\cdot z)(\eta)=z(\delta^{-1}\eta).
\]
Its orbit equivalence relation is denoted by $E(\Delta,2)$, and $\mu_\Delta$ denotes Bernoulli product measure on $2^\Delta$.  The free part of the action is invariant and conull.

We use the following consequence of Popa cocycle superrigidity, in the form isolated by Thomas and used by Smythe.

\begin{theorem}[Thomas {\cite[Theorem~3.6]{Thomas}}]\label{thm:Thomas}
Let
\[
 \Delta=\SL_3(\Z)\times S,
\]
where $S$ is any countable group.  Suppose that a countable group $\Gamma$ acts freely in a Borel way on a standard Borel space $Y$, with orbit relation $E^Y_\Gamma$.  If there is a $\mu_\Delta$-nontrivial Borel homomorphism
\[
 E(\Delta,2)\longrightarrow E^Y_\Gamma,
\]
then there is a homomorphism $\pi\colon\Delta\to\Gamma$ with finite kernel.
\end{theorem}

We shall use the following elementary consequence of Fubini's theorem.

\begin{lemma}\label{lem:symmetrization}
Let $(X,\nu)$ be a standard probability space, let $F$ be a Borel equivalence relation on $Y$, and let $g\colon X\times X\to Y$ be Borel.  Assume that
\begin{enumerate}
\item for $\nu$-almost every $x$,
\[
 g(x,u)\mathrel{F}g(x,v)
\]
for $\nu^2$-almost every $(u,v)$;
\item
\[
 g(x,u)\mathrel{F}g(u,x)
\]
for $\nu^2$-almost every $(x,u)$.
\end{enumerate}
Then there is a single $F$-class containing $g(x,u)$ for $\nu^2$-almost every $(x,u)$.
\end{lemma}

\begin{proof}
For $\nu^4$-almost every $(x,u,y,v)$, the hypotheses give
\[
 g(x,u)\mathrel{F}g(x,y)
 \mathrel{F}g(y,x)
 \mathrel{F}g(y,v).
\]
It follows that for $\nu^2\times\nu^2$-almost every pair $(p,q)\in(X\times X)^2$, one has $g(p)\mathrel{F}g(q)$.  By Fubini, some $p_0\in X\times X$ satisfies $g(p)\mathrel{F}g(p_0)$ for $\nu^2$-almost every $p$.
\end{proof}

\section{A perfect family of marked groups}

Let $F_2=\langle a,b\rangle$ be the free group on two generators.  The space
\[
 \mathcal N_2=\{N\leq F_2:N\text{ is normal}\}
\]
is a closed subspace of $2^{F_2}$, and hence a compact Polish space.  Recall that a $2$-marked group is a group equipped with an ordered generating pair.  Thus an element $N\in\mathcal N_2$ codes the $2$-marked group $(F_2/N;(aN,bN))$.

The group $F_2$ has continuum many normal subgroups~\cite[III.C.40]{deLaHarpe}, so $M$ regards $\mathcal N_2$ as an uncountable compact Polish space.  We choose the perfect family by a Cantor scheme so that its injectivity is absolute to the ambient universe.

\begin{lemma}\label{lem:absolute-perfect-family}
There is a continuous map
\[
 \vartheta\colon2^\omega\longrightarrow\mathcal N_2,
 \qquad x\longmapsto N_x,
\]
whose code belongs to $M$ and whose ambient interpretation is injective.  If $x$ is random over $M$, then $N_x\in M[x]$.
\end{lemma}

\begin{proof}
Fix in $M$ an enumeration of $F_2$ and identify $2^{F_2}$ with $2^\omega$.  Let $S\subseteq2^{<\omega}$ be the closed tree whose branches are the characteristic functions of normal subgroups of $F_2$.  Concretely, a finite binary string belongs to $S$ when it violates none of the subgroup or normality requirements whose relevant coordinates have already been decided.  Thus the definition of $S$ is absolute and
\[
 [S]=\mathcal N_2
\]
in every ambient universe containing the code.

Since $M$ regards $[S]$ as uncountable, the perfect set theorem in $M$ gives a perfect subtree $T\subseteq S$.  Perfectness of a tree is a statement about finite nodes and finite extensions, so the same $T$ is perfect in the ambient universe.  Working in $M$, choose nodes $(t_s)_{s\in2^{<\omega}}$ of $T$ so that
\begin{enumerate}
\item $t_s$ is a proper initial segment of both $t_{s^\frown0}$ and $t_{s^\frown1}$;
\item $t_{s^\frown0}$ and $t_{s^\frown1}$ are incompatible;
\item $|t_s|\geq |s|$.
\end{enumerate}
For an arbitrary ambient real $x$, let $N_x$ be the subset of $F_2$ whose characteristic function is
\[
 \bigcup_n t_{x\restr n}.
\]
This union is a branch through $T$, hence through $S$, so $N_x\in\mathcal N_2$.  Incompatibility of the two immediate successors makes $x\mapsto N_x$ injective, and the length condition makes it continuous.  The sequence $(t_s)$ belongs to $M$.  Hence, when $x$ is random over $M$, the displayed union belongs to the generic extension $M[x]$.
\end{proof}

Put
\[
 L_x=F_2/N_x.
\]
If $x$ is random over $M$, then $L_x\in M[x]$.

\begin{lemma}\label{lem:target-avoidance}
For every countable group $\Gamma$ in the ambient universe, the set
\[
 A_\Gamma=\{x\in2^\omega:L_x\text{ embeds into }\Gamma\}
\]
is countable.
\end{lemma}

\begin{proof}
Suppose that $j\colon L_x\hookrightarrow\Gamma$ is an embedding.  The ordered pair
\[
 \bigl(j(aN_x),j(bN_x)\bigr)\in\Gamma^2
\]
determines a homomorphism $\varphi\colon F_2\to\Gamma$.  Since $j$ is injective and $L_x=F_2/N_x$, one has $\ker(\varphi)=N_x$.  Thus a fixed ordered pair in $\Gamma^2$ can witness the embeddability of at most one $L_x$, because $x\mapsto N_x$ is injective.  Since $\Gamma^2$ is countable, so is $A_\Gamma$.
\end{proof}

For a countable group $L$, define
\[
 S_L=L*\Z,
 \qquad
 \Delta_L=\SL_3(\Z)\times S_L.
\]

\begin{lemma}\label{lem:no-finite-normal}
For every countable group $L$, the group $\Delta_L$ has no nontrivial finite normal subgroup.
\end{lemma}

\begin{proof}
First, $\SL_3(\Z)$ has no nontrivial finite normal subgroup.  Indeed, let $K\trianglelefteq\SL_3(\Z)$ be finite and let $g\in K$.  The conjugacy class of $g$ is finite, so its centralizer has finite index.  For each elementary one-parameter subgroup
\[
 U_{ij}=\{I+nE_{ij}:n\in\Z\},\qquad i\neq j,
\]
the intersection of $U_{ij}$ with the centralizer of $g$ has finite index in $U_{ij}$.  Hence $g$ commutes with $I+mE_{ij}$ for some nonzero $m$, and therefore with $E_{ij}$.  This holds for every $i\neq j$.  A matrix commuting with all elementary matrices is scalar, while the center of $\SL_3(\Z)$ is trivial.  Thus $g=I$.

Next let $K$ be a finite normal subgroup of $L*\Z$.  By the fixed-point theorem for finite groups acting on trees, applied to the Bass--Serre tree of the free product, every finite subgroup of $L*\Z$ is conjugate into one of the two factors~\cite[I.6.5]{Serre}.  It cannot be conjugate into the $\Z$-factor unless it is trivial.  After conjugating, suppose that $K\leq L$.  If $1\neq k\in K$ and $t$ is a generator of the $\Z$-factor, normality gives $tkt^{-1}\in K\leq L$, contradicting the normal-form theorem for free products.  Hence $K$ is trivial.

Finally, the coordinate projections of a finite normal subgroup of
$\SL_3(\Z)\times(L*\Z)$ are finite normal subgroups of the two factors.  Both projections are trivial, and hence so is the original subgroup.
\end{proof}

\section{Bernoulli shifts in a second random extension}

For a Borel set $C\subseteq2^\omega$ and $x\in2^\omega$, write
\[
 C_x=\{u\in2^\omega:I(x,u)\in C\}.
\]
After fixing the first random real $x$, the next proposition realizes the Bernoulli shift of a group that is countable in $N[x]$ by means of a second random real.  We include a prescribed conull section for use in Theorem~\ref{thm:main}.

\begin{proposition}\label{prop:second-coordinate}
Let $N$ be a countable transitive model of a sufficiently large finite fragment of ZFC, let $x$ be random over $N$, and let $\Delta\in N[x]$ be a group which $N[x]$ regards as countably infinite.  Suppose that $C\subseteq2^\omega$ is Borel and $\mu(C_x)=1$.  Then there are an invariant conull Borel set $R\subseteq2^\Delta$ and an injective Borel map
\[
 c\colon R\longrightarrow C\cap\Rand(N)
\]
such that
\begin{enumerate}
\item $R$ is contained in the free part of the Bernoulli shift $\Delta\curvearrowright2^\Delta$;
\item $N[c(z)]=N[x,z]$ for every $z\in R$;
\item $c(z)\,E_N^{\B}\,c(\delta\cdot z)$ for every $z\in R$ and $\delta\in\Delta$.
\end{enumerate}
Consequently, there is a Borel homomorphism from $E(\Delta,2)\restr R$ to $E_N^{\B}\restr(C\cap\Rand(N))$.
\end{proposition}

\begin{proof}
Choose in $N[x]$ a bijection $e\colon\omega\to\Delta$ and let
\[
 e^*\colon2^\Delta\longrightarrow2^\omega,
 \qquad e^*(z)(n)=z(e(n)),
\]
be the induced measure-preserving homeomorphism.  Put
\[
 D=\{z\in2^\Delta:I(x,e^*(z))\in C\}.
\]
Since $\mu(C_x)=1$, the set $D$ is conull.  Put
\[
 D_0=\bigcap_{\delta\in\Delta}\{z\in2^\Delta:\delta\cdot z\in D\}.
\]
Since $\Delta$ is countable, $D_0$ is Borel, invariant, and conull.

Let $R$ be the intersection of $D_0$ with the set of points of $2^\Delta$ which are random over $N[x]$ for Bernoulli product measure.  Externally $N[x]$ is countable, so $R$ is Borel and conull.  It is invariant because every shift and its inverse are measure-preserving Borel maps coded in $N[x]$.  The free part of the Bernoulli action is an $N[x]$-coded conull Borel set, and therefore every point random over $N[x]$ belongs to it.

For $z\in R$, define
\[
 c(z)=I(x,e^*(z)).
\]
The map $c$ is Borel and injective, and $c(z)\in C$ because $z\in D$.  Since $e^*$ is a measure-preserving Borel isomorphism coded in $N[x]$, the real $e^*(z)$ is random over $N[x]$.  By Lemma~\ref{lem:two-step}, $c(z)$ is random over $N$.  Since $I^{-1}$ recovers $x$ and $e^*(z)$, while $e\in N[x]$, one has
\[
 N[c(z)]=N[x,e^*(z)]=N[x,z].
\]
For $\delta\in\Delta$, the shift $z\mapsto\delta\cdot z$ and its inverse belong to $N[x]$.  Hence
\[
 N[x,z]=N[x,\delta\cdot z],
\]
and therefore $N[c(z)]=N[c(\delta\cdot z)]$.
\end{proof}

\begin{remark}
Smythe's Bernoulli embedding starts with a countable group already in $M$~\cite[Theorem~4.3]{Smythe}.  Proposition~\ref{prop:second-coordinate} allows the group to belong to the first random extension $N[x]$.  Its Bernoulli shift is then realized using a second random real, while the resulting extension is still generated over $N$ by one random real.
\end{remark}

\section{Homomorphisms into essentially free relations}

We first record the elementary ergodicity of random generic equivalence.

\begin{lemma}\label{lem:E-ergodic}
The relation $\EBM$ is ergodic with respect to $\mu$.
\end{lemma}

\begin{proof}
The finite-change relation $E_0$ on $2^\omega$, restricted to $\Rand(M)$, is a subequivalence relation of $\EBM$: if $x$ and $y$ differ in only finitely many coordinates, then $M[x]=M[y]$.  Moreover, $\Rand(M)$ is invariant under finite changes.  Thus every $\EBM$-invariant Borel subset of $\Rand(M)$, viewed as a subset of $2^\omega$, is $E_0$-invariant.  The Kolmogorov zero--one law gives measure zero or one.
\end{proof}

We begin with the conull case.

\begin{proposition}\label{prop:conull-ergodicity}
Let $C\subseteq\Rand(M)$ be conull and Borel.  If $F$ is an essentially free countable Borel equivalence relation, then $(\EBM\restr C,\mu)$ is $F$-ergodic.
\end{proposition}

\begin{proof}
It is enough to treat the case in which $F=E^Y_\Gamma$ is the orbit relation of a free Borel action of a countable group $\Gamma$.  Indeed, if $F$ is essentially free, compose the given homomorphism with a Borel reduction of $F$ to such a free orbit relation.  Triviality of the composite implies triviality of the original homomorphism.

Let
\[
 f\colon \EBM\restr C\longrightarrow E^Y_\Gamma
\]
be a Borel homomorphism.  Fix $y_0\in Y$ and extend $f$ to a Borel map $\widetilde f\colon2^\omega\to Y$ by setting $\widetilde f(w)=y_0$ off $C$.  Define
\[
 g(x,u)=\widetilde f(I(x,u)).
\]
We verify the two hypotheses of Lemma~\ref{lem:symmetrization}.

By Fubini's theorem and Lemma~\ref{lem:target-avoidance}, the set
\[
 X_0=
 \{x\in\Rand(M):\mu(C_x)=1\text{ and }x\notin A_\Gamma\}
\]
is conull.  Fix $x\in X_0$ and form
\[
 \Delta_x=\Delta_{L_x}
 =\SL_3(\Z)\times(L_x*\Z).
\]
In $M[x]$, the group $\Delta_x$ is countably infinite (indeed finitely generated).  Since $L_x$ embeds into $\Delta_x$, the group $\Delta_x$ does not embed into $\Gamma$.  By Lemma~\ref{lem:no-finite-normal}, $\Delta_x$ has no nontrivial finite normal subgroup.

Apply Proposition~\ref{prop:second-coordinate} to $N=M$, the real $x$, the group $\Delta_x$, and the set $C$.  Let $R_x\subseteq2^{\Delta_x}$ and $c_x\colon R_x\to C$ be as supplied there.  Then
\[
 \psi_x=f\circ c_x\colon R_x\longrightarrow Y
\]
is a Borel homomorphism from $E(\Delta_x,2)\restr R_x$ to $E^Y_\Gamma$.  Since $R_x$ is invariant, extend $\psi_x$ to all of $2^{\Delta_x}$ by assigning the value $y_0$ off $R_x$.

This extension must be $\mu_{\Delta_x}$-trivial.  Otherwise Theorem~\ref{thm:Thomas} would yield a homomorphism
\[
 \pi\colon\Delta_x\longrightarrow\Gamma
\]
with finite kernel.  Its kernel is a finite normal subgroup of $\Delta_x$, hence trivial by Lemma~\ref{lem:no-finite-normal}.  Thus $\pi$ would embed $\Delta_x$ into $\Gamma$, contradicting $x\notin A_\Gamma$.

Let $e_x\colon\omega\to\Delta_x$ be the bijection used in Proposition~\ref{prop:second-coordinate}.  The triviality just proved yields a conull Borel set $Q_x\subseteq R_x$ whose image under $\psi_x$ lies in one $E^Y_\Gamma$-class.  Since $e_x^*$ is measure preserving, $U_x=e_x^*[Q_x]$ is conull in $2^\omega$, and
\[
 g(x,u)\mathrel{E^Y_\Gamma}g(x,v)
\]
whenever $u,v\in U_x$.  This holds for every $x\in X_0$, so the first hypothesis of Lemma~\ref{lem:symmetrization} is satisfied.

For the second hypothesis, consider a product-random pair $(x,u)$ over $M$ such that both $I(x,u)$ and $I(u,x)$ belong to $C$.  Such pairs form a conull subset of $2^\omega\times2^\omega$.  Both interleavings are random over $M$, and
\[
 M[I(x,u)]=M[x,u]=M[I(u,x)].
\]
Since $f$ is a homomorphism,
\[
 g(x,u)=f(I(x,u))
 \mathrel{E^Y_\Gamma}
 f(I(u,x))=g(u,x).
\]
Thus the second hypothesis of Lemma~\ref{lem:symmetrization} also holds.

The lemma gives a single $E^Y_\Gamma$-class containing $g(x,u)$ for almost every pair $(x,u)$.  Since $I$ is measure preserving and $C$ is conull, a conull subset of $C$ is mapped by $f$ into that class.  Hence $f$ is $\mu$-trivial.
\end{proof}

\begin{proof}[Proof of Theorem~\ref{thm:main}]
Let $A\subseteq\Rand(M)$ be Borel with $\mu(A)>0$, and let
\[
 f\colon\EBM\restr A\longrightarrow F
\]
be a Borel homomorphism, where $F$ is essentially free.

Let $[A]_{\EBM}$ be the $\EBM$-saturation of $A$.  By the Lusin--Novikov uniformization theorem~\cite[Theorem~18.10]{Kechris}, the saturation is Borel and there is a Borel map
\[
 r\colon[A]_{\EBM}\longrightarrow A
\]
such that $r(x)\,\EBM\,x$ for every $x\in[A]_{\EBM}$.  Since $[A]_{\EBM}$ is $\EBM$-invariant and contains the positive-measure set $A$, Lemma~\ref{lem:E-ergodic} shows that it is conull.  The map $r$ is a homomorphism from $\EBM\restr[A]_{\EBM}$ to $\EBM\restr A$.

By Proposition~\ref{prop:conull-ergodicity}, the composite $f\circ r$ maps a conull subset $Z$ of $[A]_{\EBM}$ into a single $F$-class.  If $a\in A\cap Z$, then $a\,\EBM\,r(a)$ and both points belong to $A$, so
\[
 f(a)\mathrel{F}f(r(a)).
\]
It follows that $f$ maps the $\mu_A$-conull set $A\cap Z$ into that same $F$-class.  Thus $(\EBM\restr A,\mu_A)$ is $F$-ergodic.
\end{proof}

\begin{proof}[Proof of Corollary~\ref{cor:positive}]
Suppose that
\[
 f\colon\EBM\restr A\longrightarrow F
\]
is a weak Borel reduction, where $\mu(A)>0$ and $F$ is essentially free.  By Theorem~\ref{thm:main}, a $\mu_A$-conull subset of $A$ is mapped into one $F$-class.  That class is countable, while a countable-to-one map has countable preimage of a countable set.  This contradicts the nonatomicity of $\mu_A$.

If $\EBM\restr A$ were essentially free, then it would be Borel-reducible, and hence weakly Borel-reducible, to the orbit relation of a free Borel action.  Thus no positive-measure restriction is essentially free.
\end{proof}

\begin{remark}
Theorem~\ref{thm:main} says, for example, that every Borel homomorphism from any positive-measure restriction of $\EBM$ to a hyperfinite or treeable countable Borel equivalence relation is measure-trivial, since such relations are essentially free~\cite{KechrisMiller}.  Thus the conclusion is stronger than the corresponding non-amenability and non-treeability statements.
\end{remark}

We finish with an immediate consequence for iterations ending in random forcing.  Use the standard Borel coding of a $\mathbb P*\dot\B$-generic filter by a pair $(G,w)$, where $G$ is $\mathbb P$-generic over $M$ and $w$ is random over $M[G]$.  Under this coding, every fixed-$G$ fiber is a Borel subspace.

\begin{corollary}\label{cor:random-tail}
Let $\mathbb P\in M$, let $\dot\B$ be the canonical $\mathbb P$-name for random forcing in the $\mathbb P$-extension, and fix a $\mathbb P$-generic filter $G$ over $M$.  Put $N=M[G]$.  For every Borel set $A\subseteq\Rand(N)$ with $\mu(A)>0$ and every essentially free countable Borel equivalence relation $F$, the restriction of $\eqgen{M}{\mathbb P*\dot\B}$ to
\[
 \{(G,w):w\in A\}
\]
is $F$-ergodic with respect to normalized random measure $\mu_A$ on the fixed-$G$ fiber.  Consequently, for every essentially free countable Borel equivalence relation $F$,
\[
 \eqgen{M}{\mathbb P*\dot\B}\not\leBw F.
\]
\end{corollary}

\begin{proof}
On the fixed-$G$ fiber, for $w,w'\in\Rand(N)$ one has
\[
 (G,w)\eqgen{M}{\mathbb P*\dot\B}(G,w')
 \quad\Longleftrightarrow\quad
 M[G,w]=M[G,w']
 \quad\Longleftrightarrow\quad
 N[w]=N[w'].
\]
Thus the restriction of $\eqgen{M}{\mathbb P*\dot\B}$ to this fiber is precisely $E_N^{\B}$.  The first assertion is Theorem~\ref{thm:main}, applied over the countable transitive model $N$.

For the consequence, fix an essentially free countable Borel equivalence relation $F$ and suppose that a weak Borel reduction from $\eqgen{M}{\mathbb P*\dot\B}$ to $F$ existed.  Restrict it to the fixed-$G$ fiber over $\Rand(N)$.  Applying the first assertion with $A=\Rand(N)$ makes this restriction measure-trivial.  But an $F$-class is countable, and the preimage of a countable set under a countable-to-one map is countable, contradicting the nonatomicity of random measure.
\end{proof}

\begin{remark}
The same conclusion holds whenever, below some condition in $M$, a forcing is equivalent to an iteration ending in random forcing.
\end{remark}

\end{document}